\documentclass[10pt, reqno]{amsart}
\usepackage{amssymb,amsmath,amsthm}

\usepackage{amsfonts}

\usepackage{latexsym}
\usepackage{color}

\usepackage{cite}

\numberwithin{equation}{section}

\usepackage[pdftex,linkcolor=blue,citecolor=blue]{hyperref} 

\def\pa{\partial}
\def\sdg{Schr\"odinger }

\newcommand{\R}{\mathbb{R}}

\newcommand{\norm}[1]{\left\| #1\right\|}
\newcommand{\eps}{\varepsilon}

\newcommand{\al}{\alpha}

\newcommand{\les}{\lesssim}
\newcommand{\ges}{\gtrsim}

\newcommand{\vp}{\varphi}

\newcommand{\DO}{\Delta_\Omega}

\newcommand{\tq}{\tilde{q}}
\newcommand{\tr}{\tilde{r}}

\renewcommand{\Re}{\text{ Re }}
\renewcommand{\Im}{\text{ Im }}

\newcommand{\ra}{\rightarrow}

\newcommand{\ep}{\epsilon}

\newcommand{\tsf}{C_c^\infty(\Omega)}

\newtheorem{theorem}{Theorem}[section]

\newtheorem{lemma}[theorem]{Lemma}

\newtheorem{corollary}[theorem]{Corollary}

\newtheorem{proposition}[theorem]{Proposition}

\theoremstyle{definition}
\newtheorem{definition}{Definition}[section]
\newtheorem{remark}[theorem]{Remark}

\newcommand{\beq}{\begin{equation}}
\newcommand{\eeq}{\end{equation}}

\newcommand{\beqq}{\begin{equation*}}
\newcommand{\eeqq}{\end{equation*}}

\theoremstyle{remark}

\makeatletter
\newcommand{\Extend}[5]{\ext@arrow0099{\arrowfill@#1#2#3}{#4}{#5}}
\makeatother

\begin{document}
\title[Scattering for focusing NLS on exterior domain]{
Scattering
for 3d cubic focusing NLS on the domain
outside a convex obstacle revisited}

\author{Chengbin Xu}%
\address{School of Mathematics and Statistics,
 Zhengzhou University,
 100 Kexue Road, Zhengzhou, Henan, 450001, China
}%
\email{xcbsph@163.com}%

\author{Tengfei Zhao}%
\address{Beijing Computational Science Research Center,
No. 10 West Dongbeiwang Road, Haidian District, Beijing, China, 100193 }
\email{zhao\underline{ }tengfei@csrc.ac.cn}%

\author{Jiqiang Zheng}
\address{Institute of Applied Physics and Computational Mathematics,
P. O. Box 8009,\ Beijing,\ China,\ 100088}
\email{zhengjiqiang@gmail.com}

\begin{abstract}
In this article, we consider the
focusing cubic nonlinear
\sdg equation(NLS)
in the exterior domain outside of a convex obstacle  in $\R^3$ with Dirichlet
 boundary conditions. We revisit the scattering result below ground state in  Killip-Visan-Zhang\cite{KVZ-2016-AMR}
by utilizing the method of Dodson and Murphy \cite{DM-2017-PAMS,DM-2017-np-nonradial} and the dispersive estimate in Ivanovici and Lebeau \cite{IL-2017-CRM},
which avoids using the concentration compactness.
We conquer the difficulty of the boundary in the focusing case by establishing
a local smoothing effect of the boundary.
Based on this effect and
the interaction Morawetz estimates, we prove the solution decays
at a large time interval, which meets the scattering criterions.

\end{abstract}

 \maketitle

\begin{center}
 \begin{minipage}{100mm}
   { \small {{\bf Key Words:}  Schr\"odinger equation; exterior domain; global well-posedness;  scattering criterions.}
      {}
   }\\
    { \small {\bf AMS Classification:}
      {35P25,  35Q55, 47J35.}
      }
 \end{minipage}
 \end{center}



\section{Introduction}

Consider the Cauchy problem of the nonlinear Schr\"odinger
 equation with Dirichlet boundary condition
\beq\label{nls-ex-do}
\left\{
\begin{aligned}
i\pa_t u+\Delta u&=\,-|u|^{2}u=:F(u), \quad (t,x)\in \R\times \Omega \\
  u(0,x)&=\,\phi(x),\\
  u(t,x)&=\,0, \quad \quad \quad x\in \pa \Omega,
 \end{aligned}
\right.
\eeq
where  $\Omega$ is the exterior of a
smooth, compact, strictly convex obstacle
 $\Omega^c \subset \R^3$ with smooth boundary $\pa \Omega$, and
$\Delta$ is the Dirichlet Laplacian operator. 
It is easy to find that the solution $u$ to
equation \eqref{nls-ex-do}
with sufficient smooth conditions posses
the energy conservation
\beq\label{E-Om}
E_{\Omega}(u(t))~:=~\int_\Omega \left[ \frac12|\nabla u(t,x)|^2 -\frac1{4} |u(t,x)|^{4} \right]dx=E_{\Omega}(u_0)
\eeq
and mass conservation
\beq\label{M-Om}
M_{\Omega}(u(t)) ~:= ~\int_\Omega |u(t,x)|^2 dx=M_{\Omega}(u_0).
\eeq

When $\Omega=\R^3$, the Cauchy
problem
\begin{equation}\label{equ:nls}
  \begin{cases}
  i\pa_tu+\Delta u+|u|^2u=0,\quad (t,x)\in\R\times\R^3,\\
  u(0,x)=u_0(x),
  \end{cases}
\end{equation}
is scale invariant. More precisely,
the class of solutions to \eqref{equ:nls} is left invariant by the
scaling
\begin{equation}\label{scale}
u(t,x)\mapsto \lambda u(\lambda^2t, \lambda
x),\quad\lambda>0.
\end{equation}
Moreover, one can also check that the only homogeneous $L_x^2$-based
Sobolev space that is left invariant under \eqref{scale} is
$\dot{H}_x^\frac12(\R^3).$
Hence, we say that the Cauchy problem \eqref{nls-ex-do}
 is $\dot H^{\frac12}$-critical.
We will consider the well-posedness and long time behavior of the
Cauchy problem \eqref{nls-ex-do} with initial data in the energy spaces.
To do it, we first recall the classical Sobolev spaces on the domain $\Omega$.
\begin{definition}
 For integer $k\geq1$ and $1\leq p\leq \infty$, we denote  $H^{k,p}_0(\Omega)$
as the closure of  $C_c^\infty(\Omega)$  under the norm
$$\norm{u}_{H^{k,p}_0(\Omega)}~:=~ \sum_{|\al|\leq k}  \norm{\pa^\al u}_{L^p(\Omega)}.  $$
If $p=2$, we also write $H^{k}_0(\Omega)=H^{k,2}_0(\Omega)$ for simplicity.

\end{definition}

In fact, $-\Delta$ is an unbounded and positive semi-define symmetric operator on $\tsf$.
We define the corresponding quadratic form by for $u,v\in \tsf$
\beqq
Q(u,v)~=~\int_\Omega  \nabla u(x) \nabla \bar{v}(x) dx.
\eeqq
The extension of form $Q$ is unique and defined on $H^1_0(\Omega)$.
Then the Friedrichs extension
of $-\Delta$ gives
the Dirichlet Laplacian on $\Omega$, $-\DO$, which is a
 self-adjoint operator and with form domain $Q(-\DO)=D(\sqrt{-\DO})$.
By the spectral theorem, we are able to denote the spectral measure $E(\lambda)$  and the operators 
by
\beqq
\vp(\sqrt{-\DO})= \int_{[0,\infty)} \vp(\lambda) d E(\lambda).
\eeqq
Thus, the linear operator $e^{it\DO}$ associated
to the free \sdg  equation on $\Omega$ is well
defined and unitary on $L^2(\Omega)$. And we can define the
Sobolev spaces based on the operator $\DO$.

\begin{definition}
For $s\geq 0$ and $1<p<\infty$, let $\dot H^{s,p}_D(\Omega)$ and    $H^{s,p}_D(\Omega)$
denote the completions of $C_c^\infty(\Omega)$ under the norms

$$    \norm{f}_{\dot H^{s,p}_D(\Omega)}:= \norm{(-\DO)^\frac{s}2 f}_{L^p(\Omega)} \quad { \text{and} } \quad
\norm{f}_{ H^{s,p}_D(\Omega)}:= \norm{(1-\DO)^\frac{s}2 f}_{L^p(\Omega)}.
$$
When $p=2$ we also write  $\dot H^{s}_D(\Omega)$  and $ H^{s}_D(\Omega)$ for $\dot H^{s,2}_D(\Omega)$
 and $ H^{s,2}_D(\Omega)$, respectively.
\end{definition}
These two definitions are equivalent under certain conditions, see Proposition \ref{eq-Sobolev} below.

%
%


For the  Euclidean space $\R^d$,
the linear operator $ e^{it\Delta}$ obeys the dispersive estimates and the Strichartz estimates.
Owing to this, the local well-posedness theory of the
solutions to equation \eqref{equ:nls} with the general
power type nonlinearities $F(u)=|u|^{p-1}u$ is standard.
For the defocusing energy subcritical($F(u)=-|u|^{p-1}u,~~1+\frac{4}{d}<p<1+\frac{4}{d-2}$) cases, the solutions with initial datum in $H^1(\R^d)$ are global well-posed and scatter,
see \cite{Cazenave-2003} \cite{killip2009nonlinear} and references therein.

In general domains, we do not have the dispersive estimate and the Strichartz estimates
for  $e^{it\Delta_\Omega}$.
For the case of exterior domain of a convex obstacle, Ivanovici
\cite{Ivanovici-2010-Apde} proved the Strichartz
estimates except endpoint case by using the Melrose and Taylor parametrix
 and she also proved  the scattering theory energy
 subcritical NLS for exterior domain of smooth convex obstacle in $3D$.
Ivanovici and Lebeau \cite{IL-2017-CRM} proved the dispersive estimates holds only in the $3D$ case.
 For more  scattering results of defocusing subcritical NLS in the general exterior domains, we refer to Planchon-Vega\cite{PV-2009-Bilinear},
Ivanovici-Planchon \cite{IP-2010-Poincare}, and Blair-Smith-Sogge\cite{BSS-2008-PAMS}.

In this paper, we consider scattering theory of the solutions to focusing equation \eqref{nls-ex-do}, which is mass supercritical and energy subcritical.
In fact, the nonlinear elliptic equation
\beq\label{ellip-eq}
-\Delta \vp +\vp= |\vp|^2\vp,
\eeq
has infinite number of solutions in $H^1(\R^3)$.
Then for any solution $\vp \in H^1(\R^3)$  to \eqref{ellip-eq}, $e^{it}\vp$ is a global and  non-scattering solution to the Cauchy problem \eqref{equ:nls}.
Furthermore, there exists a minimal mass solution and we often denote it as $Q$ and call it the ground state, which is  positive, radial, exponentially decaying, see Cazenave\cite{Cazenave-2003} and Tao\cite{Tao2006}.
Holmer-Roudenko\cite{HR-2008-CMP} proved  the global well-posedness and scattering theory for radial solutions to equation \eqref{equ:nls} such the following conditions in $\R^3$:
\begin{gather}\label{A}
E_{\R^3}(u_0)M_{\R^3}(u_0) < E_{\R^3}(Q)M_{\R^3}(Q), \tag{A}
\end{gather}
\begin{gather}\label{B}
 \norm{\nabla u_0}_{L^2(\R^3)} \norm{u_0}_{L^2(\R^3)}< \norm{\nabla Q}_{L^2(\R^3)} \norm{Q}_{L^2(\R^3)}. \tag{B}
\end{gather}
Duyckaerts-Holmer-Roudenko \cite{DHR-2008-MRL} removed the radial assumption.
Killip-Visan-Zhang \cite{KVZ-2016-AMR} proved the results for  exterior domains of convex obstacles  in $\R^3$:

\begin{theorem}\label{main-thm}
Let $ \Omega$ is  exterior of a convex obstacle in $\R^3$.  
If the initial data $u_0\in H^1_D(\Omega)$ satisfies
\begin{align}\label{equ:energy}
 E_{\Omega}(u_0)M_{\Omega}(u_0) <& E_{\R^3}(Q)M_{\R^3}(Q), \\\label{equ:kinerng}
\|\nabla u_0\|_{L^2(\Omega)} \|u_0\|_{L^2(\Omega)}<& \norm{\nabla Q}_{L^2(\R^3)} \norm{Q}_{L^2(\R^3)},
\end{align}
then,
the corresponding  solution to the Cauchy problem \eqref{nls-ex-do} with initial $u_0$ is globally well-posed and scatters.
\end{theorem}
The proofs of \cite{DHR-2008-MRL} and \cite{KVZ-2016-AMR} utilized the  concentration-compactness arguments basing on the profile decomposition
introduced by Kenig-Merle \cite{KM-2006-Invent,kenig2008global}, which have become powerful and effective methods for many dispersive equations and many  other equations.

In this article, we revisit  Theorem
\ref{main-thm}, by employing an
%
%
%
%
%
idea of  Dodson-Murphy \cite{DM-2017-PAMS},
\cite{DM-2017-np-nonradial}, which provide
new proofs in the Euclidean case avoiding uses of
concentration and compactness.

{\bf Outline of proof:}
By the Strichartz estimates and  the equivalence of various Sobolev norm definitions, we have the local well-posedness of  \eqref{nls-ex-do} in $H^1_D(\Omega)$.
From the coercivity property(Lemma \ref{Coercivity} below) under the ground state,
we know the solution $u$ is globally well-posed and of bounded
$H^1_D(\Omega )$ norm.
Utilizing the dispersive estimates, we prove that
the scattering criterion given by
\cite{DM-2017-np-nonradial} also holds in our case, that is:
if for any large time window, there exists a
large subinterval such that a space-time norm of
of $u$ is  small in it, then $u$ must scatter.

%


To end the proof, the main difficulties are how
to overcome the effect from  boundary $\pa\Omega$ and
the lack of the Galilean invariance.
Combining with the concavity of  $\pa\Omega$ and the coercivity property,
the Morawetz estimates yields a weaker local smoothing effect 
on the boundary.
On the other hand, as in  \cite{DHR-2008-MRL}, for the Euclidean case,
by the Galilean invariance, one can
assume the critical solution $u_c$
has zero conserved  momentum, which yields the
spatial translation parameter 
$x(t)=o(t)$  (as $ t \ra \infty$).
This fact is essential to the preclusion of
the critical solution by making use of the Morawetz estimates centered at origin.
For our case, the momentum is obvious
bounded since $u\in L^\infty_tH^1_x$.
Based on this fact, one could just expect $|x(t)|\les |t|.$
However, the interaction Morawetz
identity is defined as an average of the Morawetz action
that is centered any point in $\R^3$.
Fortunately, since $u\in L^\infty_tH^1_x$, we are able to prove the
smallness $L^3_{t,x}$-norm
in a large subinterval of any large time
interval without employing the Galilean transformation.

Finally, this and a standard continuity argument
imply the solution such that the conditions of the scattering criterion.

%

\begin{remark}

Our proof is based on the
the dispersive estimates of \cite{IL-2017-CRM}, which
does not hold true in higher dimensions.
Nevertheless, in these cases, it is hopeful that one may prove
the corresponding results 
via establishing weaker dispersive estimates(see for example \cite{Zheng-2018-JMP}).

\end{remark}

\begin{remark}
We remark that  the interaction Morawetz estimates
also reflect that the solution decays in big ball
around any point.
In fact, for any fixed $R>0$, we have
\beqq
\liminf_{t\ra\infty} \sup_{x(t)\in\R^3}\norm{u(t,\cdot-x(t))}_{L^2_x(\Omega\cap B_R)} ~= ~0,
\eeqq
where $B_r$ is the ball center at origin with radius $r$.
This suffices the scattering criterion for non-radial NLS \eqref{nls-ex-do} in \cite{Tao2004DPDE} when $\Omega =\R^3$.
\end{remark}

\begin{remark}
In fact, as in  \cite{Tao-2005-NYJM}  and \cite{DM-2017-np-nonradial},
one can check that our proof would imply
\beqq
\norm{u}_{L^5_{t,x}(\R\times\Omega)} \les \exp\{ \exp{A(E(u_0),M(u_0))} \},
\eeqq
where $A$ is a  rational  polynomial of $E(u_0),M(u_0)$ and $ E(Q),M(Q)$.
The double-exponential growth derives  from the local smoothing
effect of boundary and the interaction Morawetz estimates.
\end{remark}

\begin{remark}
Our arguments can be used to prove the similar results for general focusing energy subcritical  cases($F(u)=-|u|^{p-1}u$, $\frac73<p<5$), which has been considered in \cite{Kai-2017-CPAA}.

\end{remark}

This article is organized as follows: in Section 2,
 we recall some basics facts on the domain.
Section 3 is devoted to prove the scattering under the
 assumption of smallness of $L^5_{t,x}$ norm of the solution.
In Section 4, we verify the scattering criterion.

We conclude the introduction by giving some notations which
will be used throughout this paper. We always use $X\lesssim Y$ to denote $X\leq CY$ for some constant $C>0$. $X \sim Y$  stands for $X\lesssim Y$ and $Y\lesssim X$.  Similarly, $X\lesssim_{u} Y$ indicates there exists a constant $C:=C(u)$ depending on $u$ such that $X\leq C(u)Y$.
The symbol $\nabla$ refers to the spatial derivation.
For $M=\R^3$ or a domain in $\R^3$, we use $L^r(M)$ to denote the Banach space
of functions $f:M\rightarrow\mathbb{C}$ whose norm
$$\|f\|_{L^r(M)}=\Big(\int_{M}|f(x)|^r dx\Big)^{\frac1r}$$
is finite, with the usual modifications when $r=\infty$.
For a time slab $I$, we use $L_{t}^qL^r_x(I\times M)$ to denote the space-time norm 
\begin{align*}
  \|f\|_{L_{t}^qL^r_x(I\times M)}=\bigg(\int_{I}\|f(t,x)\|_{L^r_x(M)}^q dt\bigg)^\frac{1}{q}
\end{align*}
with the usual modifications when $q$ or $r$ is infinite. 

\section{Basic tools and the local theory}

In this section we give some basic harmonic tools and the local well-posedness theory
for the Cauchy problem \eqref{nls-ex-do}. In this section, we assume that $\Omega$ is the complement of a compact convex body $\Omega^c\subset \R^3$ with smooth boundary.

First, we recall the following proposition.
\begin{proposition}[Equivalence of the Sobolev norms, \cite{KVZ-2015-IMRN}]\label{eq-Sobolev}

Let $1<p<\infty $. If $0\leq s <\min\{1+\frac1p,\frac3p \}$, then

\begin{equation}
\norm{(-\Delta_{\R^3})^\frac{s}{2} f}_{L^p(\R^3)}
\sim_{p,s} \norm{(-\DO)^\frac{s}{2} f}_{L^p(\Omega)}
\end{equation}
for all $f\in C_c^\infty(\Omega)$.

\end{proposition}

Using this proposition, we have
\begin{corollary}[Fractional product rule, \cite{KVZ-2015-IMRN}]
 For all $f,g\in C_c^\infty(\Omega)$, we have

 \beqq
 \norm{(-\Delta_\Omega)^{\frac{s}{2}}(fg) }_{L^p(\Omega)} \les
 \norm{(-\Delta_\Omega)^{\frac{s}{2}} f}_{L^{p_1}(\Omega )} \norm{g}_{L^{p_2}(\Omega)} +
 \norm{f}_{L^{q_1(\Omega)}} \norm{(-\Delta_\Omega)^\frac{s}{2} g}_{L^{q_2}(\Omega)}
 \eeqq
  with the exponents satisfying $1<p,p_1,q_2 \leq \infty$, $1<p_2,q_1\leq \infty$,
  \beqq
  \frac1p=\frac1{p_1}+\frac1{p_2}=\frac1{q_1}+\frac1{q_2}, \,\,\text{ and }  \,\, 0<s<\min\left\{1+\frac1{p_1},1+\frac1{q_2},\frac{3}{p_1},\frac{3}{q_2}\right\}.
  \eeqq

\end{corollary}

\begin{corollary}[Fractional chain rule,\cite{KVZ-2015-IMRN}]\label{fcr-KVZ}
Suppose $G\in C^1(\mathbb{C})$, $s\in (0,1]$, and $1<p,p_1,p_2<\infty$ are such that
$\frac1p=\frac{1}{p_1}+\frac{1}{p_2}$ and $0<s <\min\left\{ 1+\frac{1}{p_2},\frac3{p_2}\right\}$. Then

\beqq
\norm{(-\Delta_\Omega)^{\frac{s}{2} } G(f)}_{L^p(\Omega)} \les_{s,p,p_1} \norm{G'(f)}_{L^{p_1}(\Omega)} \norm{(-\DO)^{\frac{s}{2}} f}_{L^{p_2} }.
\eeqq
\end{corollary}

We need the chain rule for fractional derivatives on $\R^d$, which will be useful for
the local theory.
\begin{proposition}[Chain rule for fractional derivatives,\cite{KPV-1993-CPAM}]\label{fcr-KPV}
  If $F\in C^2$, with $F(0)=0,F'(0)=0,$ and $|F''(a+b)|\leq C\{ |F''(a)| +|F''(b)| \}$, and $|F'(a+b)|\leq
  C\{ |F'(a)| +|F'(b)| \},$ we have, for $0<\al<1,$

  \beqq
  \norm{ \Lambda^\al
     F(u) }_{L^p_x(\R^d)} \leq C \norm{F'(u)}_{L^{p_1}(\R^d)}
  \norm{\Lambda^\al  u }_{L^{p_2}(\R^d)}, \, \frac{1}{p}=\frac{1}{p_1} +\frac{1}{p_2},
  \eeqq
and
  \beqq
    \begin{split}
  &    \norm{\Lambda^\al   [F(u)-F(v)] }_{L^p_x(\R^d)} \\
      \leq  ~~& \,C [\norm{F'(u)}_{L^{p_1}(\R^d)}  +\norm{F'(v)}_{L^{p_1}(\R^d)}
         ] \norm{\Lambda^\al (u-v) }_{L^{p_2}(\R^d)}\\
      &+ C \left\||F''(u)|+|F''(v)|\right\|_{L^{r_1}(\R^d)}
      \left(\norm{\Lambda^\al u }_{L^{r_2}(\R^d)}+\norm{ \Lambda^\al  v }_{L^{r_2}(\R^d)} \right) \norm{(u-v) }_{L^{r_3}(\R^d)},
    \end{split}
  \eeqq
where $\Lambda=(-\Delta_{\R^3})^{\frac12}. $
\end{proposition}

Next, we recall the dispersive estimates.


\begin{lemma}[Dispersive estimate,\cite{IL-2017-CRM}]\label{disp-est}
\beq
\norm{e^{it\Delta_\Omega}f }_{L^\infty_x(\Omega)} \les  t^{-\frac{3}{2} }  \norm{f}_{L_x^1(\Omega)}.
\eeq
\end{lemma}
Combining this  with the endpoint Strichartz estimate  of Keel-Tao, we have the following Strichartz estimates:
\begin{proposition}[Strichartz estimates \cite{Ivanovici-2010-Apde}\cite{keel1998endpoint}]
 Let   $q,\tq\geq2$, and
 $2\leq r,\tr\leq \infty$ satisfying
 $\frac{3}{2}=\frac2q+\frac{3}{r}=\frac{2}{\tq}+\frac{3}{\tr}$.
 Then,  the solution $u$ to $(i\partial_t+\Delta)u = F$
 on an interval $I\ni 0$  satisfies
 \begin{equation}\label{equ:stri}
   \|u\|_{L^q_tL^r_x(I\times \Omega)}\les
 \norm{u_0}_{L^2(\Omega)} +
 \norm{F}_{ L^{\tq'}_tL^{\tr'}_x(I\times \Omega)}.
 \end{equation}

\end{proposition}

%

We define the  $S(I)$ and $W(I)$ norm for a interval $I$ by
\beq
\norm{u}_{S(I)}=\norm{u}_{L^5_{t,x}(I\times \Omega )}   \quad \text{ and } \quad
\norm{u}_{W(I)}=\norm{u}_{L^5_t L^{\frac{30}{11}}_x(I\times \Omega )}.
\eeq
Note that $1< \min\{1+\frac{11}{30},\frac{11}{10}\}$.
Thus, by Strichartz, Corollary \ref{fcr-KVZ}, Proposition \ref{fcr-KPV}, we have: 

\begin{theorem}[Local well-posedness, \cite{CW-1990-Nonlinear}\cite{KM-2006-Invent}]
  Assume that $u_0\in \dot H_D^{1}(\Omega), 0\in I ,$ and $\norm{u_0}_{ \dot H_D^{1}(\Omega)} \leq A.$
  Then there exists $\delta=\delta(A)$ such that if $\norm{e^{it\DO} u_0}_{S(I)}\leq \delta, $
   there exists a unique solution $u$ to \eqref{nls-ex-do} in $I\times \Omega$, with $u\in C(I; \dot H_D^{1}(\Omega) )$ such that
  \beq
\Big\|(-\DO)^{\frac12} u\Big\|_{W(I) } +\sup_{t\in I}\norm{u(t)}_{\dot H_D^{1}(\Omega)} \leq CA,\quad  \norm{u}_{S(I)}\leq 2\delta.
  \eeq
  Moreover, if $u_{0,k}\ra u_0 $ in $\dot H_D^{1}(\Omega)$, we obtain the corresponding solutions
  $u_k\ra u$ in $C(I; \dot H^{1}_D(\Omega))$.

\end{theorem}

\begin{remark}
From standard arguments, we have
 if  $u$ is a global solution and such that
  \beqq
  \norm{u}_{S(\R)} <\infty,
  \eeqq
then  $u$  scatters both directions.

\end{remark}

%
%
%
%
%
%
%

We need the following refined Gagliardo-Nirenberg inequality, which follows from
the sharp Gagliardo-Nirenberg inequality and the
Pohozaev identities of the ground state.

\begin{lemma}[Refined Gagliardo-Nirenberg inequality, \cite{DM-2017-np-nonradial}]\label{Sharp-Sobolev}
For $f\in H^1(\R^3)$ and any $\xi\in \R^3$,
\beq
\norm{f}_{L^4(\R^3)}^4 \leq \frac{4}{3} \Big(  \tfrac{\norm{f}_{L^2(\R^3)} \norm{f}_{\dot H^1_x(\R^3) }}{\norm{Q}_{L^2(\R^3)} \norm{ Q}_{\dot H^1_x(\R^3)}}\Big)~ \inf_{\xi\in\R^3}  \norm{e^{ix\xi} f }_{\dot H^1_x(\R^3)}^2.
\eeq
\end{lemma}

%

Before the end of this section, we recall the coercivity property for functions
 under the ground state $Q$(i.e., satisfying the conditions \eqref{A} and \eqref{B}).
We denote $M_{\R^3}$ and $E_{\R^3}$ as the Mass and energy on $\R^3$ respectively.

\begin{lemma}[Coercivity]\label{Coercivity}
Let $u_0\in H^1_D(\Omega)$ satisfy the conditions \eqref{equ:energy}.
If $\norm{u_0}_{L^2(\Omega)} \norm{u_0}_{\dot H^1_D(\Omega)}\leq \norm{Q}_{L^2_x(\R^3)} \norm{Q}_{\dot H^1_x(\R^3)}$, then there exists $\delta'=\delta'(\delta)>0$ so that
 \beq\label{LH-below}
\norm{u(t)}_{L^2_x(\Omega)} \norm{u(t)}_{\dot H^1_x(\Omega)}\leq (1-\delta')\norm{Q}_{L^2_x(\R^3)} \norm{Q}_{\dot H^1_x(\R^3)}
 \eeq
  holds for all $t\in I,$ where $u:I \times \Omega\ra \mathbb{C}$ is the maximal lifespan solution to \eqref{nls-ex-do}. In particular, $I=R$ and $u$ is uniformly bounded in $H^1(\Omega).$

 Moreover, for any  function $f\in H^1_D(\Omega)$  such that \eqref{LH-below}, there exists $\rho=\rho(\delta')>0$ such that
  \beq
\norm{f}_{\dot H^1_x(\Omega)}^2-\frac34\norm{f}^4_{L^4_x(\Omega)}  \geq \rho ( \norm{f}_{\dot H^1_x(\Omega)}^2 +\norm{f}_{L^4_x(\Omega)}^4 ) .
 \eeq
\end{lemma}

\begin{proof}
   The proof follows from Proposition  \ref{eq-Sobolev} above, Lemma 2.3 and Lemma 2.4 in \cite{DM-2017-PAMS}.

\end{proof}

\begin{remark}\label{R:coercive} Suppose $u_0\in H^1_D(\Omega)$ satisfies \eqref{equ:energy} and \eqref{equ:kinerng}.  Then by the above lemma, the maximal-lifespan solution $u$ to \eqref{nls-ex-do} with initial data $u_0$ obeys \beq\label{LH-below}
\norm{u(t)}_{L^2_x(\Omega)} \norm{u(t)}_{\dot H^1_x(\Omega)}\leq (1-\delta')\norm{Q}_{L^2_x(\R^3)} \norm{Q}_{\dot H^1_x(\R^3)}
 \eeq for all $t$ in the lifespan of $u$.  In particular, $u$ remains bounded in $H_D^1(\Omega)$ and hence is global.
\end{remark}

\section{Scattering criterion}
In this section, we prove a scattering criterion for solutions of
  the Cauchy problem \eqref{nls-ex-do}.

\begin{proposition}\label{sct-cri}
Suppose that $u$ is a global solution to \eqref{nls-ex-do}, satisfying
\beq
\norm{u}_{L^\infty_tH^1_D(\R\times \Omega) } \leq E.
\eeq
There exist $\ep=\ep(E,\Omega)>0$ and $T_0=T_0(\ep,E,\Omega)>0$
satisfying that  if for any $a\in\R$ there exists $T\in\R$ such that
$[T-\ep^{-5}, T ]  \subset (a,a+T_0)$ and
\beq\label{ST-decay}
\norm{u}_{L^5_{t,x} ([T-\ep^{-5},T)\times \Omega )}  ~\leq~ \ep,
\eeq
then $u$ scatters forward in time.
\end{proposition}
\begin{proof}
By the Strichartz estimates and continuity method, there exists $\eps= \eps(E,\Omega)$
such that if  for any $T>0$,
\beq\label{scatter-T}
\big\| e^{i(t-T)\DO} u(T) \big\|_{L^5_{t,x}([T,\infty)\times \Omega)} \leq \eps,
\eeq
then the $u$ scattering forward.

By the Duhamel formula, we have
\beq
e^{i(t-T)\DO} u(T)=e^{it\DO} u_0+i \int_0^T e^{i(t-s)\DO} (|u|^2u)(s)ds.
\eeq
First, by the Strichartz estimates, there exists $T_1>0$ such that, if $T>T_1$
\beq\label{ST-linear}
\norm{e^{it\DO} u_0}_{L^5_{t,x}([T,\infty)\times \Omega)}<\frac12 \eps.
\eeq
Take $a=T_1$, $\ep=\eps^2$,  $T$ as in the assumption \eqref{ST-decay} and make a decomposition $$[0,T]=[0,T-\ep^{-5}]\cup [T-\ep^{-5},T]   :=I_1\cup I_2.$$ 
Then by \eqref{ST-decay}, the Strichartz estimates, and the continuity method, we have
\beqq
\norm{u}_{L^5_t  H_D^{\frac{30}{11},3}([T-\ep^{-5},T]\times \Omega)} \les 1.
\eeqq
 Thus, we have
\beq\label{I-2}
\Big\|\int_{I_2} e^{i(t-s)\DO} (|u|^2u)(s) ds \Big\|_{L^5_{t,x}([T,\infty)\times \Omega)}
\les ~\norm{u}^2_{L^5_{t,x}(I_2\times \Omega)}  \norm{u}_{L^5_t  H_D^{\frac{30}{11},3}(I_2\times \Omega)} \les \ep^2.
\eeq

Next, we consider the corresponding contribution of $I_1$.
By the Duhamel formula and the Strichartz estimates, we have
\begin{align*}
  & \Big\|\int_{I_1}  e^{i(t-s)\DO} (|u|^2u)(s) ds \Big\|_{L^5_tL^\frac{30}{11}_x([T,\infty)\times \Omega)} \\
 =  & \Big\| e^{i(t-(T-T_0^\frac13))\DO} u(T-T_0^\frac13) -e^{it\DO} u_0 \Big\|_{L^5_tL^\frac{30}{11}_x([T,\infty)\times \Omega)} \les 1.
\end{align*}
On the other hand, employing the dispersive estimates and the Sobolev embedding, we have
\begin{align*}
  &\Big\|\int_{I_1}  e^{i(t-s)\DO} (|u|^2u)(s) ds \Big\|_{L^5_tL^\infty_x([T,\infty)\times \Omega)}\\
 \les ~&~ \Big\| \int_{I_1} \tfrac{1}{(t-s)^{\frac32}}  ds \Big\|_{ L^5\left([T,\infty)\right) }\norm{u}_{L^\infty_t H^1_D(\R\times \Omega)}^3 \les \ep^{\frac32} .
\end{align*}
Thus,  by interpolation,
we have
\beqq
~\Big\|\int_{I_1}  e^{i(t-s)\DO} (|u|^2u)(s) ds \Big\|_{L^5_tL^5_x([T,\infty)\times \Omega)}
\les ~\ep^{\frac{15}{22}},
\eeqq
which together with \eqref{ST-linear} and \eqref{I-2} implies \eqref{scatter-T}.
Therefore, we complete the proof.

\end{proof}

\section{Proof of Theorem \ref{main-thm}}

In this section, we prove Theorem \ref{main-thm}.
First we prove a local smoothing effect property on the Boundary $\pa\Omega$ by utilizing  a Morawetz-type estimate.
Then we prove the interaction Morawetz estimates for the solution in the Theorem \ref{main-thm}.
Finally, we prove Theorem \ref{main-thm} by showing the solution such that the conditions of the scattering criterion in previous section.

Let $\chi_R(x)$ be a smooth function on $\R^3$ and such that  $\chi_R(x)=1$ when $|x|\leq \frac{R}{4}$ and  $\chi_R(x)=0$ when $|x|\geq \frac{R}2.$
We need the following coercivity property, which follows similar proof of  Lemma 3.2 in \cite{DM-2017-PAMS}.
\begin{lemma}[Coercivity on balls]\label{coer-2}
 There exists $R=R(\delta,M(u),Q)>0$ sufficiently large such that for any point $z\in\R^3$,
\beq
\sup_{t\in \R} \norm{\chi_R(\cdot-z) u(t)}_{L^2_x(\Omega)}
 \norm{\chi_R(\cdot-z) u(t)}_{\dot H^1_x(\Omega)} <(1-\delta)
  \norm{Q}_{L^2_x(\R^3)}\norm{Q}_{\dot H^1_x(\R^3)}
\eeq
In particular, by Lemma \ref{Coercivity}, there exists $\delta'=\delta'(\delta)>0$ so that
\beq
\norm{\chi_R(\cdot-z) u(t)}_{\dot H^1_x(\Omega)}^2 -\frac{3}{4}\norm{\chi_R(\cdot-z) u(t)}^4_{L^4_x(\Omega)} \geq
\delta'  \norm{\chi_R(\cdot-z) u(t)}_{\dot H^1_x(\Omega)}^2
\eeq
uniformly for $t\in \R.$
\end{lemma}

Next, we make some preparation for the Morawetz estimates.
Let $n(x)$ be the outer normal vector at $ x\in \pa \Omega$ and define the outer derivative by $\pa_n f=  \triangledown f\cdot n.$ Denote $dS$ be the induced measure on $\pa\Omega$.

Let $\eta>0$ small, $\chi(x)=1 $ for $|x|\leq 1-\eta$ and $\chi=0$ for $x\geq1$.
  Let $R>1$ large, and define
  \beqq
  \phi(x)=\frac{1}{\omega_3 R^3} \int_{\R^3} \chi^2(\tfrac{x-s}{R})\chi^2(\tfrac{s}{R}) ds,
  \eeqq
  and
    \beqq
  \phi_1(x)=\frac{1}{\omega_3 R^3} \int_{\R^3} \chi^2(\tfrac{x-s}{R})\chi^4(\tfrac{s}{R}) ds,
  \eeqq
where $\omega_3$ is the volume of unit ball in $\R^3.$
Then we have
\beqq
|\phi-\phi_1|\les \eta.
\eeqq
Let
\beqq
\psi(x)=\frac{1}{|x|} \int_0^{|x|} \phi(r)dr,
\eeqq
which satisfies
\beqq
|\psi(x)| \leq \min\Big\{1,\frac{R}{|x|}\Big\} \text{\quad  and \quad } \pa_k\psi(x)=\frac{x_k}{|x|^2}[\psi(x)-\phi(x)].
\eeqq
One can also deduce that
\beq\label{nabla-psi}
\pa_k[\psi(x)x_k]=3\phi(x)+2(\psi-\phi)(x),
\eeq
where the repeated indices are summed.

\subsection{Local smoothing effect}

We define the Morawetz action  by 
\beq
M(t)=2\Im \int_\Omega \psi(x) x [\bar u \nabla u ] dx.
\eeq
Then, $|M(t)|\les R$.
\begin{proposition}
For  large $T_0>1$ and  any time interval $I=[a,a+T_0]\subset\R$, we have
  \beq
  \frac{1}{T_0 }  \int_I \int_{\pa\Omega}  |\pa_n u|^2 (t,x) dS(x) dt \les
\tfrac{1}{(\log T_0)^\frac12} .
  \eeq



\end{proposition}
\begin{proof}

From the identity
\beq\label{moment_t}
2\pa_t\Im(\bar{u}u_k) ~=~\pa_k |u|^4+\pa_k \Delta |u|^2 -4  \pa_j\Re(\bar{u}_j u_k),
\eeq
 \eqref{nabla-psi}, and integration by parts, we have
\begin{align}
 &\pa_t M(t) \nonumber \\
=~& 4  \int_\Omega \Big[\phi(x) |\nabla u|^2(t,x) - \frac34 \phi_1(x) |u|^4(t,x) \Big]dx \label{M-1} \\
    & -2\int_{\pa \Omega} \psi(x) x\cdot n(x) |\pa_n u|^2 dS(x)  +  4 \int_\Omega (\psi-\phi)(x) |\not \nabla u|^2(t,x)   dx  \label{M-2}\\
    &- \int_\Omega [3(\phi- \phi_1)+2(\psi-\phi)](x)|u|^4(t,x) dx + \int_\Omega \nabla [3\phi+2(\psi-\phi)](x) \cdot \nabla |u|^2(t,x) dx, \label{M-3}
\end{align}
where 
$\not \nabla$ is the angular derivation centered at the origin.

By the definition of $\chi$,  \eqref{M-1} equals
\beq\label{coer-term}
\frac{4}{R^3}\int_{\R^3} \int_\Omega |\nabla \big(\chi(\tfrac{x-s}{R}) u \big)|^2 -\frac34 |\chi(\tfrac{x-s}{R}) u|^4  dx \chi^2(\tfrac{s}{R}) ds + O(\tfrac{1}{\eta^2 R^2} ).
\eeq
By the Coercivity property Lemma \ref{coer-2},
there exists $R_1>0$, such that the first term of \eqref{coer-term} is nonnegative for $R>R_1$.
And the nonnegativity  for second term of  \eqref{M-2} follows from  the fact $\psi-\phi\geq0.$
From the facts $\phi-\phi_1 \les \eta$ and
\beq\label{est-phipsi}
|\psi-\phi|+|\nabla \phi|+ |\nabla \psi|\les |\psi-\phi|(1+\frac{1}{|x|}) +|\nabla \phi|
\leq \frac{1}{\eta R} +\min\Big\{ \frac{|x|}{\eta R}, \frac{R}{|x|}\Big\} +\min\Big\{\frac{1}{\eta R}, \frac{R}{|x|^2}\Big\},
\eeq
we have
\beq
\frac{1}{J} \int_{R_0}^{e^J R_0}  |\phi-\phi_1|+|\psi-\phi|+|\nabla \phi|+ |\nabla \psi|  \frac{dR}{R}
 \les \eta +\frac{1}{J\eta} +\frac{1}{R_0 \eta J}.
\eeq
Thus, we can deduce that
\beq
-\frac{1}{T_0}\int_I  \frac{1}{J} \int_{R_0}^{e^J R_0} \int_{\pa\Omega}  \psi(x) x\cdot n(x) |\pa_n u|^2 (t,x) dx  \frac{dR}{R}dt \les
 \eta +\frac{1}{J\eta} +\frac{1}{R_0 \eta J}+\frac{R_0 e^J}{T_0J}+ \frac{1}{\eta^2 J R_0^2 } .
  \eeq
Since the boundary $\pa\Omega$ of $\Omega $ is concave and compact, we have $-\psi(x)x\cdot n(x)=x \cdot n(x)\ges 1$ for $x\in\pa\Omega$,  which yields
 \beq
 \frac{1}{T_0}\int_I \int_{\pa\Omega}  |\pa_n u|^2 (t,x) dx  dt \les
  \frac{1}{J \eta R_0 }+ \frac{1}{J\eta}+ \eta +\frac{R_0 e^J}{T_0J } + \frac{1}{\eta^2 J R_0^2  }.
\eeq
Then the conclusion follows by taking
$ \eta=R_0^{-1}=J^{-\frac12}=(\log T_0)^{-\frac12}$.

\end{proof}

\subsection{Interaction Morawetz estimates}

We define the interaction Morawetz quantity
\beq
M_R(t)=2\iint_{\Omega\times\Omega} |u|^2(t,y) \psi(x-y) (x-y) \Im [\bar{u}\nabla u](t,x) dxdy,
\eeq
which reflects  the information of $u$ on whole $\Omega.$
One can easily find that for any $R>0$ and $t\in \R$, $$|M_R(t)| \les R E_0^2.$$

\begin{theorem}[Interaction Morawetz estimates]
For arbitrary small $\eps>0$, there exists $T_0, R_0>0$ large  and $\eta>0$  small enough
satisfying that: for any interval $I=[a,a+T_0]$, there exists $\xi=\xi(s,t,R) \in \R^3$
such that
\beq\label{inter-morawetz}
\frac{1}{JT_0}\int_{R_0}^{R_0e^J} \int_I  \frac{1}{R^3}\int_{\R^3}\iint_{\Omega\times \Omega }
       \left|\chi\big(\tfrac{\cdot-s}{R}\big) u\right|^2(t,y) \left|\nabla( \chi\big(\tfrac{\cdot-s}{R}\big) u^\xi)\right|^2(t,x) dxdyds dt \frac{dR}{R}\les \eps.
\eeq

\end{theorem}

\begin{proof}
  By the identities \eqref{moment_t} and
 \begin{align}
\pa_t |u|^2 &~=~-2 \pa_k\Im(\bar{u} u_k),
 \end{align}
  we have
  \begin{align}
  \pa_tM_R(t) = &  \int\int_{\Omega\times \Omega  }  |u|^2(t,y) \psi(x-y) (x-y)\nabla |u|^4(t,x) dxdy \label{IM-1}\\
       & + \int\int_{\Omega\times \Omega  }  |u|^2(t,y) \psi(x-y) (x-y) \nabla \Delta |u|^2(t,x) dxdy\label{IM-2}\\
       & -4 \int\int_{\Omega\times \Omega  }  |u|^2(t,y) \psi(x-y) (x_k-y_k) \Re( \pa_j (\bar u_j u_k)(t,x) dxdy \label{IM-3}\\
      & -4 \int\int_{\Omega\times \Omega  }  \pa_j \Im (\bar u u_j)(t,y)  \psi(x-y)(x-y)_k \Im(\bar u u_k)(t,x) dxdy. \label{IM-4}
  \end{align}
 By integration by parts and the Dirichlet boundary condition of $u$, we have
  \begin{align}
 \eqref{IM-1}=& -\iint_{\Omega\times \Omega }  |u|^2(t,y) [3\phi(x-y)+2(\psi-\phi)(x-y)] |u|^4(t,x) dxdy \nonumber \\
            =& -3\iint_{\Omega\times \Omega }  |u|^2(t,y) \phi_1(x-y)|u|^4(t,x) dxdy \label{IM-1-1} \\
            & -2\iint_{\Omega\times \Omega }  |u|^2(t,y) (\psi-\phi)(x-y) |u|^4(t,x) dxdy  \label{IM-1-2}\\
            &-3\iint_{\Omega\times \Omega }  |u|^2(t,y) (\psi-\phi_1)(x-y) |u|^4(t,x) dxdy. \label{IM-1-3}
  \end{align}
Here, we view  \eqref{IM-1-2} and \eqref{IM-1-3} as error terms from the definitions the cutoff functions.

   \begin{align}
 \eqref{IM-2}=
            ~& \iint_{\Omega\times \Omega }  |u|^2(t,y) \nabla_x\left[3\phi(x-y)+2(\psi-\phi)(x-y)\right]  \nabla_x\left[|u|^2(t,x)\right] dxdy \label{IM-2-1}\\
            & +2 \int_\Omega \int_{\pa \Omega}  |u|^2(t,y)   \psi(x-y)  (x-y) \vec n_x |\pa_nu|^2(t,x) dS(x)dy. \label{IM-2-2}
  \end{align}
As above, we also regard \eqref{IM-2-1} as an error term. We will apply the local smoothing effect to
the estimation of \eqref{IM-2-2}
     \begin{align}
 \eqref{IM-3}~ =~ & 4 \iint_{\Omega\times \Omega }  |u|^2(t,y)  \phi(x-y) |\nabla u|^2(t,x) dxdy \label{IM-3-1}\\
            & +4 \iint_{\Omega\times \Omega }
            |u|^2(t,y)  P_{ij}(x-y) (\psi-\phi)(x-y) \Re[\bar u_j u_k]  dxdy \label{IM-3-2}\\
            &-4\int_\Omega \int_{\pa \Omega}  |u|^2(t,y) \psi(x-y)  (x-y)_k \Re (\pa_n \bar u u_k)(t,x)  dS(x)dy.  \label{IM-3-3}
  \end{align}

  \begin{align}
 \eqref{IM-4} ~=~ & -4 \int\int_{\Omega\times \Omega  }  \pa_{y_j} \Im (\bar u u_j)(t,y)  \psi(x-y)(x-y)_k \Im(\bar u u_k)(t,x) dxdy \nonumber \\
  =  & -4 \iint_{\Omega\times \Omega}  \phi(x-y)  \Im (\bar u \nabla u)(t,y)   \Im (\bar u \nabla u)(t,x) dxdy  \label{IM-4-1}\\
   & -4  \iint_{\Omega\times \Omega}  \Im (\bar u \nabla u_j)(t,y)  P_{jk}(x-y) [\psi(x-y)-\phi(x-y)] \Im (\bar u \nabla u_k)(t,x)      dxdy, \label{IM-4-2}
  \end{align}
where $P_{ij}(x)=\delta_{ij}-\frac{x_ix_j}{|x|^2}.$

 From  the fact that $\psi-\phi\geq 0$ and Cauchy-Schwarz, we have
  \begin{align}
      & \eqref{IM-3-2}+ \eqref{IM-4-2} \nonumber \\
    = & 4\iint_{\Omega\times\Omega} |u|^2(t,y) |\not\nabla_y u(t,x) |^2  [(\psi-\phi)(x-y)] dxdy \nonumber \\
    & -4\iint_{\Omega\times\Omega}
    \Im [\bar u\not\nabla_xu ](t,y)
    \Im [\bar u\not\nabla_yu ](t,x)  [(\psi-\phi)(x-y)] dxdy \geq 0,
   \end{align}
  where $\not \nabla_z$ is the angular derivation centered at $z\in\R^3$.
By the compactness and convexity of $\pa \Omega$, we have
 \begin{align}
    & |\eqref{IM-2-2}+ \eqref{IM-3-3}|  \nonumber\\
   =~&  \left|2  \int_\Omega \int_{\pa \Omega} |u|^2(t,y)\psi(x-y)  (x-y)  n(x) |\pa_nu|^2(t,x) dxdy\right|\\
   \les~ &   R \int_\Omega \int_{\pa \Omega} |u|^2(t,y)   ~|\pa_nu|^2(t,x) dxdy. \nonumber
\end{align}
By a direct computation, one has
  \begin{align*}
   & \frac{\omega_3R^3}{4}[\eqref{IM-3-1}+ \eqref{IM-4-1}] \\
    = &\int_{\R^3}\iint_{\Omega\times \Omega }
    \chi^2\big(\tfrac{x-s}{R}\big) \chi^2\big(\tfrac{y-s}{R}\big) \left[ |u|^2(t,y)|\nabla u|^2(t,x) -\Im (\bar u \nabla u)(t,y)\Im (\bar u \nabla u)(t,x)  \right] dxdyds\\
   = & \int_{\R^3}\iint_{\Omega\times \Omega }
    \chi^2\big(\tfrac{x-s}{R}\big) \chi^2\big(\tfrac{y-s}{R}\big)  |u|^2(t,y)|\nabla u^\xi|^2(t,x)   dxdyds,
\end{align*}
  for $u^\xi(t,x)=e^{ix\xi}u(t,x)$ and $$\xi(t,s,R)= -\frac{ \int_\Omega \chi^2\big(\tfrac{x-s}{R}\big) \Im (\bar u \nabla u)(t,x) dx }{ \int_\Omega \chi^2\big(\tfrac{x-s}{R}\big) |u|^2(t,x) dx} $$ or $\xi=0$ if $\int_\Omega \chi^2\big(\tfrac{x-s}{R}\big) |u|^2(t,x) dx=0$ .

Combining these estimates above, we have

\begin{align*}
& \frac{1}{R^3}\int_{\R^3}\iint_{\Omega\times \Omega }
       |\chi(\tfrac{\cdot-s}{R}) u|^2(t,y) \Big[|\nabla\big(  \chi(\tfrac{\cdot-s}{R}) u^\xi\big)|^2(t,x) -\frac{3}{4} |\chi(\tfrac{\cdot-s}{R}) u|^4 (t,x)\Big] dxdyds\\
\les & \frac{1}{\eta^2 R^2} + \pa_t M_R(t) +  \int_\Omega \int_{\pa \Omega} |u|^2(t,y)  |x \cdot n(x)| ~|\pa_nu|^2(t,x) dxdy \\
    & +\iint_{\Omega \times \Omega} |u|^2(t,y) |u|^4[ (\psi-\phi)(x-y) +(\phi-\phi_1)(x-y)] dxdy \\
    & + \iint_{\Omega \times \Omega} |u|^2(t,y) |u\nabla u|(t,x) \left| \nabla (\psi +\phi)(x-y) \right| dxdy.
\end{align*}
By the Lemma \ref{coer-2} and \eqref{est-phipsi},  for sufficiently large $R>0$,
we have
\begin{align*}
 & \frac{1}{JT_0}\int_{R_0}^{R_0e^J} \int_I  \frac{1}{R^3}\int_{\R^3}\iint_{\Omega\times \Omega }
       |\chi(\tfrac{\cdot-s}{R}) u|^2(t,y) |\nabla( \chi\big(\tfrac{\cdot-s}{R}\big) u^\xi)|^2(t,x) dxdyds dt \frac{dR}{R} \\
       \les & \frac{1}{ J \eta^2 R_0^2} + \frac{R_0 e^J}{(\log T_0)^\frac12 J } + \frac{1}{R_0 \eta J}+ \frac{1}{J\eta} +\eta ,  
\end{align*}
which implies  the conclusion  \eqref{inter-morawetz}
by taking $\eta=J^{-\frac12}=R_0^{-1} =\eps$ and  $ \log T_0=e^{\eps^{-2}}$.


\end{proof}

\subsection{Proof of Theorem \ref{main-thm}}

  By the interaction Morawetz estimates and the Sobolev embedding,
there exists $T_0>0$ and $R\in[R_0,e^J R_0]$ such that for any interval $I=[a,a+T_0]$
\beqq
\frac{1}{T_0} \int_{I} \frac{1}{R^3} \int_{\R^3} \Big\|\chi\big(\tfrac{\cdot-s }{R}\big) u(t)\Big\|_{L^2(\Omega)}^2  \Big\|\nabla \Big(\chi\big(\tfrac{\cdot-s }{R}\big) u^\xi(t) \Big)\Big\|_{L^2(\Omega)}^2
dsdt\les \eps. 
 \eeqq
Thus there exists $\theta\in[0,1]^3$ such that
\beqq
\frac{1}{T_0} \int_{I}\sum_{z\in\mathbb{Z}^3} \Big\|\chi\big(\tfrac{\cdot-\frac{R}{4}(z+\theta) }{R} \big) u(t)\Big\|_{L^2(\Omega)}^2  \Big\|\nabla \Big(\chi\big(\tfrac{\cdot-\frac{R}{4}(z+\theta) }{R}\big) u^\xi(t) \Big)\Big\|_{L^2(\Omega)}^2 dt \les \eps.
\eeqq
Therefore, there exists a subinterval $I_0=[b-\eps^{-\frac14},b]\subset I$ such that
\beqq
\int_{I_0} \sum_{z\in \mathbb{Z}^3} \Big\|\chi\big(\tfrac{\cdot-\frac{R}{4}(z+\theta) }{R}\big) u(t)\Big\|_{L^2(\Omega)}^2  \norm{\nabla \Big(\chi\big(\tfrac{\cdot-\frac{R}{4}(z+\theta) }{R} \big) u^\xi(t) \Big)}_{L^2(\Omega)}^2 dt \les \eps^\frac34.
\eeqq
This together with  the Gagliardo-Nirenberg  inequality
$$\|f\|_{L^3}^4\lesssim\|f\|_{L^2}^2\|\nabla f\|_{L^2}^2$$
 implies that
\beq\label{L34}
\int_{I_0} \sum_{z\in \mathbb{Z}^3} \Big\|\chi\big(\tfrac{\cdot-\frac{R}{4}(z+\theta) }{R}\big) u(t)\Big\|_{L^3(\Omega)}^4  dt \les  \eps^\frac34.
\eeq
On the other hand, by H\"older's inequality and Sobolev embedding, we have
\beqq
\sum_{z\in \mathbb{Z}^3} \Big\|\chi\big(\tfrac{\cdot-\frac{R}{4}(z+\theta) }{R}\big) u(t)\Big\|_{L^2(\Omega)} \Big\|\chi\big(\tfrac{\cdot-\frac{R}{4}(z+\theta) }{R}\big) u(t)\Big\|_{L^6(\Omega)} 
\les 1,
\eeqq
which yields
\beq\label{L32}
\sum_{z\in \mathbb{Z}^3} \Big\|\chi\big(\tfrac{\cdot-\frac{R}{4}(z+\theta) }{R} \big) u(t)\Big\|_{L^3(\Omega)}^2 \les 1 ,
\eeq
Now, we have, by  \eqref{L34} and\eqref{L32},
\beq
\begin{split}
&\norm{u}^3_{L^3(I_0\times\Omega)}\\
\leq~& \int_{I_0} \sum_{z\in \mathbb{Z}^3} \norm{\chi(\frac{\cdot-\frac{R}{4}(z+\theta) }{R}) u(t)}_{L^3(\Omega)}^3  dt \\
\leq ~&   \int_{I_0}  \left(\sum_{z\in \mathbb{Z}^3} \norm{\chi(\frac{\cdot-\frac{R}{4}(z+\theta) }{R}) u(t)}_{L^3(\Omega)}^4 \right)^{\frac12}   \left(\sum_{z\in \mathbb{Z}^3} \norm{\chi(\frac{\cdot-\frac{R}{4}(z+\theta) }{R}) u(t)}_{L^3(\Omega)}^2 \right)^{\frac12} dt                               \\
\leq ~&  \left(\int_{I_0} \sum_{z\in \mathbb{Z}^3} \norm{\chi(\frac{\cdot-\frac{R}{4}(z+\theta) }{R}) u(t)}_{L^3(\Omega)}^4 dt\right)^{\frac12}
  \left(\int_{I_0} \sum_{z\in \mathbb{Z}^3}  \norm{\chi(\frac{\cdot-\frac{R}{4}(z+\theta) }{R}) u(t)}_{L^3(\Omega)}^2 dt\right)^{\frac12}                               \\
 \leq & \eps^{\frac14}.
\end{split}
\eeq
By interpolation, we have
\begin{align}
\norm{u}_{L^5_{t,x}(I_0\times\Omega)}
 \leq~ \norm{u}_{L^3_{t,x}(I_0\times\Omega)}^\frac37
\norm{u}_{L^{10}_{t,x}(I_0\times\Omega)}^\frac47
 \les~  \eps^{\frac{1}{28}-\frac14\frac1{10} \frac{4}{7}}
  \les~   \eps^{\frac{3}{140}},
\end{align}
where we have used the fact that
\beqq
\norm{u}_{L^{10}_{t,x}(I\times\Omega)} \les \langle |I| \rangle^\frac{1}{10},
\eeqq
which is a direct consequence of the Strichartz estimates, the Sobolev inequality  and a standard continuity argument.
Then, by the scattering criterion in  Proposition \ref{sct-cri},
the conclusion follows.

\vskip 0.2in

%



 \noindent\textbf{Acknowledgements}.
The authors would like to thank Jason Murphy for his helpful discussions. J. Zheng  was partly supported by NSFC Grants 11771041, 11831004. This work is financially supported by National Natural Science Foundation of China (NSAF - U1530401).

\end{document}